\documentclass[dvipsnames,a4paper]{amsart}

\usepackage[dvipsnames]{xcolor} 
\usepackage{amssymb}
\usepackage{wasysym}
\usepackage{marvosym}
\usepackage{manfnt} 
\usepackage{enumitem}
\usepackage{tikz}
\usepackage{hyperref}
\usepackage{mathtools}
\usepackage{mathscinet}
\usepackage{graphicx}
\usepackage{psfrag}
\usepackage{datetime}
\usepackage{tikz-cd}
\usepackage{url}
\usepackage{adjustbox}
\usepackage[utf8]{inputenc}
\usepackage[margin=25mm]{geometry}

\usepackage{tabstackengine}
\setstackgap{L}{.75\baselineskip}

\usetikzlibrary{decorations.pathmorphing}
\usetikzlibrary{decorations.pathreplacing}
\usetikzlibrary{positioning}
\usetikzlibrary{calc}
\usetikzlibrary{intersections}
\usetikzlibrary{fit}

\makeatletter
\tikzset{
    dot diameter/.store in=\dot@diameter,
    dot diameter=3pt,
    dot spacing/.store in=\dot@spacing,
    dot spacing=10pt,
    dots/.style={
        line width=\dot@diameter,
        line cap=round,
        dash pattern=on 0pt off \dot@spacing
    }
}
\makeatother

\numberwithin{equation}{section}

\newtheorem{theorem}{Theorem}[section]
\newtheorem{lemma}[theorem]{Lemma}
\newtheorem{proposition}[theorem]{Proposition}

\theoremstyle{definition}

\newtheorem{notation}[theorem]{Notation}
\newtheorem{example}[theorem]{Example}
\newtheorem{convention}[theorem]{Convention}

\theoremstyle{remark}
\newtheorem{remark}[theorem]{Remark}

\newcommand{\st}{\mid} 

\newcommand{\nonconsec}[2]{\prescript{\circlearrowleft}{}{\mathbf{I}_{#1}^{#2}}}
\newcommand{\derset}[2]{{\tilde{\mathbf{I}}_{#1}^{#2}}}
\newcommand{\modset}[2]{{\mathbf{I}_{#1}^{#2}}}

\DeclareMathOperator{\Hom}{Hom}
\DeclareMathOperator{\Ext}{Ext}
\DeclareMathOperator{\End}{End}
\DeclareMathOperator{\modules}{mod}
\DeclareMathOperator{\additive}{add}

\DeclareFontFamily{U}{mathx}{\hyphenchar\font45}
\DeclareFontShape{U}{mathx}{m}{n}{
      <5> <6> <7> <8> <9> <10>
      <10.95> <12> <14.4> <17.28> <20.74> <24.88>
      mathx10
      }{}
\DeclareSymbolFont{mathx}{U}{mathx}{m}{n}
\DeclareFontSubstitution{U}{mathx}{m}{n}
\DeclareMathAccent{\widecheck}{0}{mathx}{"71}

\newcommand{\onevec}{\mathbf{1}} 
\newcommand{\bder}{D^{b}(\modules\Lambda)} 
\newcommand{\dla}{\argdl{A_{n}^{d}}} 
\newcommand{\argdl}[1]{\mathcal{D}_{#1}} 
\newcommand{\derclus}{\argclus{\Lambda}} 
\newcommand{\derclusa}{\argclus{A_{n}^{d}}} 
\newcommand{\argclus}[1]{\mathcal{U}_{#1}} 
\newcommand{\U}{U}

\newcommand{\argot}[1]{\otca_{#1}} 
\newcommand{\otclus}{\argot{\Lambda}} 
\newcommand{\otclusa}{\argot{A_{n}^{d}}} 
\newcommand{\otca}{\mathcal{O}} 

\newcommand{\argds}[1]{\mathcal{U}_{#1}^{\{-d, 0\}}} 
\newcommand{\dsubcat}{\argds{\Lambda}} 
\newcommand{\dsubcata}{\argds{A_{n}^{d}}} 

\newcommand{\projinj}{\mathcal{P} \cap \mathcal{I}} 
\newcommand{\quotcat}{\mathcal{B}} 
\newcommand{\shortap}{\mathcal{C}} 

\newcommand{\swr}{\taurus} 

\newcommand{\bbe}{\mathbb{E}}
\newcommand{\fs}{\mathfrak{s}}

\usepackage[backend=biber,
	style=alphabetic,
	url=false,
	doi=false,
	isbn=false,
	maxbibnames=99]{biblatex}
\DeclareFieldFormat{title}{#1} 
\renewbibmacro{in:}{} 
\title[Relations between cluster categorifications in higher type $A$]{Relations between categorifications of \\ higher-dimensional type $A$ cluster combinatorics}

\author{Mikhail Gorsky}
\email{mikhail.gorskii@univ-st-etienne.fr}
\address{Universit\"at Hamburg, Fachbereich Mathematik, Bundesstraße 55, 20146 Hamburg, Germany
\newline 
{\tiny{and}} Institut Camille Jordan UMR 5208, Université Jean Monnet, CNRS, Centrale Lyon, INSA Lyon, Université Claude Bernard Lyon 1, 20, rue Annino, 42023, Saint-Étienne, France}
\author{Nicholas J. Williams}
\email{nw480@cam.ac.uk}
\address{Department of Pure Mathematics and Mathematical Statistics, Centre for Mathematical Sciences, University of Cambridge, Wilberforce Road, Cambridge, CB3 0WB, United Kingdom}

\begin{document}

\begin{abstract}
We consider three categories arising from the higher Auslander algebras of type~$A$ in relation to $d$-dimensional cluster combinatorics: $d$-exact subcategory of the module category of $A^d_{n+1}$ generated by the $d$-cluster-tilting object, the $(d+2)$-angulated cluster category, and the $d$-almost positive subcategory of the derived category (the higher analogue of the category of two-term complexes of projectives).  
We show that the third one, introduced by the second-named author, is the $d$-exangulated quotient of the other two, introduced by Oppermann and Thomas, by the ideals generated by morphisms factoring through morphisms from injective to projective objects, thus providing an algebraic connection between the two models 
of Oppermann-Thomas.
This is a $d$-exangulated version in type $A$ of a result of Brüstle and Yang and its interpretation by the first-named author together with Fang, Palu, Plamondon and Pressland.
It also explains a well-known coincidence between the number of 2-term silting complexes in type $A_{n}$ and of tilting modules in type $A_{n+1}$ from the $0$-Auslander perspective.  
We expect this to serve as a prototypical example of 
the $0$-Auslander correspondence in higher homological algebra.
\end{abstract}

\maketitle

\section{Introduction}

Around 25 years ago, motivated by problems related to dual canonical bases in Lie theory and the concept of total positivity in matrix groups and flag varieties, Fomin and Zelevinsky \cite{FZ} introduced the notion of cluster algebras.
Cluster algebras are constructed somewhat differently to the usual method of defining algebras via generators and relations.
A cluster algebra is always contained in the intersection of infinitely many Laurent polynomial rings, defined by a recursive procedure via so-called mutations; it many natural situations, it coincides with this intersection. In the algebro-geometric interpretation of this, the spectrum of a cluster algebra is covered up to 
codimension 2 by overlapping spectra of Laurent polynomial rings, i.e., by algebraic tori.

Since their introduction, cluster algebras and their geometric incarnations turned out to appear in many
contexts which might look completely unrelated to each other at a quick glance, underscoring their significance. This includes algebraic combinatorics \cite{CFZ}, representation theory of finite-dimensional algebras \cite{AIR, DF}, Calabi-Yau differential graded categories \cite{amiot}, Lie theory \cite{gls-survey}, integrable systems \cite{B, Keller}, Donaldson-Thomas theory in algebraic geometry \cite{nagao}, physics of scattering amplitudes \cite{amplyt}, study of moduli spaces of local systems of surfaces \cite{FG}, low-dimensional contact geometry \cite{CGGLSS, CW}, various aspects of mirror symmetry \cite{GHKK, KeelYu, CGSS, RW, AB}, and many other areas besides. 

The paper concerns with categorification of higher-dimensional cluster combinatorics in the context of higher homological algebra.
Higher homological algebra has been a subject of an intensive research over the last 20 years. It originates from a series of works of Iyama introducing higher Auslander--Reiten theory and higher-dimensional version of Auslander correspondence \cite{iya-mos, iya-Ac, iy-clus}. Since then, it has been connected to such diverse topics as symplectic geometry and Fukaya categories of symmetric products of discs \cite{DJL}, higher Waldhausen construction in algebraic $K$-theory \cite{jasso_k}, study of singularities \cite{HIMO}, and homological minimal model program in noncommuative algebraic geometry \cite{JKM1,JKM2}, where it was used to prove the Donovan-Wemyss  conjecture stating that smooth irreducible 3-fold flops are classified by their contraction algebras \cite{August, DW}.

Categorification of higher-dimensional cluster combinatorics goes back to the seminal paper by Oppermann and Thomas \cite{ot} which presents two different models related to triangulations of cyclic polytopes: the subcategory of a module category generated by a $d$-cluster-tilting object and a $(d+2)$-angulated category.

We relate the two categorical models of \cite{ot} to each other. The relation goes through an additional category and is motivated by two results on categorifications of usual cluster combinatorics, one applying only in type $A$ and the other in a more general context.

One can categorify clusters and their mutations via tilting modules over the path algebra of the corresponding quiver, in linearly oriented type $A$ \cite{BK}, or via cluster-tilting objects in a $2$-Calabi-Yau algebraic triangulated category, in general \cite{BMRRT}. In linearly oriented type $A$, these two frameworks are in fact related via a third one: that of silting complexes in the category of two-term complexes of projectives, up to homotopy \cite{AI}.

All the categories involved can be naturally considered as extriangulated categories \cite{NP1}. In fact, all the categories can be endowed with $0$-Auslander extriangulated structures, in the sense of \cite{GNP2}. These are extriangulated structures of projective dimension at most $1$ and dominant dimension at least $1$ or, equivalently, of injective dimension at most $1$ and codominant dimension at least $1$. In such categories, the notions of silting, tilting, and cluster-tilting agree under reasonable conditions, and there is a common theory of mutations of indecomposable summands. We refer to \cite{GNP2} for all the definitions and precise statements.

With this in mind, the statements in this classical setting are as follows.

\begin{enumerate}
    \item  The category of two-term complexes of projective modules, up to homotopy, over the path algebra of the linearly oriented $A_n$ quiver is the quotient of the module category of the linearly oriented $A_{n+1}$ quiver by the ideal of morphisms factoring through the unique indecomposable projective-injective module $P$.
 We first learnt this from a similar observation of Ringel \cite{ringel_catalan}, thought it is possible that this was independently observed elswhere.
 We explain this fact from the perspective of $0$-Auslander extriangulated categories 
 In fact, every morphism from a projective to an injective object in the latter module category factors through $P$, and so this ideal can be alternatively defined as the ideal generated by morphisms factoring first through an injective and then through a projective object. 

\item The category of two-term complexes of projective modules, up to homotopy, over a Jacobian algebra of a quiver with potential is a quotient of the corresponding cluster category by a certain ideal of morphisms. This was first observed in \cite{BrustleYang}. Their statement was explained and generalised in \cite{FGPPP} from the following perspective. One considers a relative $0$-Auslander extriangulated structure, which first appeared in \cite{PPPP},  on the cluster category. With this structure on the domain of the quotient functor, the ideal of morphisms turns out to be generated by morphisms factoring first through an injective and then through a projective object. The proofs both in \cite{BrustleYang} and \cite{FGPPP} use dg enhancements. The latter interpretation was further incorporated by Chen into a 0-dimensional version of the Auslander correspondence for exact dg categories \cite{chen20230, Chen24}; these are dg enhancements of extriangulated categories \cite{chenI}.
\end{enumerate}

Some ingredients have been generalised to the higher-dimensional context. 

\begin{enumerate}
 \item 
 Replacement of module category of the linearly oriented $A_n$ quiver is given in \cite{ot} by the subcategory of the module category of the higher Auslander algebra of type $A$, $A^d_n$, additively generated by the unique $d$-cluster-tilting object $M^{(d,n)}$.

\item 
The usual cluster category is replaced by a higher-dimensional analogue \cite{ot} which carries a $(d+2)$-angulated structure in the sense of \cite{gko}.

 \item Extriangulated categories have been generalised to $d$-exangulated categories by Herschend-Liu-Nakaoka \cite{hln}. A relevant relative $d$-exangulated structure on the 
cluster category, generalising the one from \cite{PPPP}, first appeared in \cite{JS}, where it was motivated by the study of indices. 
\end{enumerate}

We combine these ingredients together to give higher-dimensional counterparts of the aforementioned results on quotients of categories in the usual cluster setting. To be illustrative,  we deliberately avoid working with general $d$-exangulated categories in the main part of the paper and restrict to (mostly combinatorial) considerations in higher type $A$.

We discuss in detail the $d$-exangulated category serving as an analogue of the category of 2-term complexes of projectives.
It first appeared in \cite{njw-icra}, and was used as a third model for the combinatorics of triangulations of cyclic polytopes in \cite{njw-phd}.
Explicitly, this category is defined as $\additive (M^{(d, n)} \oplus A_{n}^{d}[d])$, where $M^{(d, n)}$ is the basic $d$-cluster-tilting object of $A_{n}^{d}[d]$, which is unique up to isomorphism.
We call it the \emph{$d$-almost-positive category}.
We prove the following two results.

\begin{theorem} (= Theorem~\ref{thm:equiv})
\label{thm:intro_1}
As a $d$-exangulated category, the $d$-almost positive category in type $A^d_n$ is the quotient of the (unique) $d$-cluster tilting subcategory of the module category of $A^d_{n+1}$ by the ideal of morphisms factoring through projective-injective objects.
\end{theorem}

\begin{theorem} (=Theorem~\ref{thm:main2})
\label{thm:intro_2}
As a $d$-exangulated category, the $d$-almost positive category in type $A_n^d$ is the quotient of the cluster category, endowed with the relative structure of \cite{JS} induced by $A^d_n$ considered as a cluster-tilting object, 
by the ideal generated by morphisms with injective domain and projective codomain.
\end{theorem}

We note that the quotients by morphisms with injective domain and projective codomain as above naturally preserve and reflect rigidity and 
exchange $d$-exangles, just as explained in \cite{FGPPP} in the case $d = 1$.
This explains why different models categorify the same combinatorics: the analogues of two-term silting objects and their mutations in the $d$-almost-positive category match with the tilting modules and their mutations under the quotient in Theorem~\ref{thm:intro_1} and with the maximal rigid objects and their mutations in the cluster category (considered with the relative structure) under the quotient in Theorem~\ref{thm:intro_2}. The latter coincide with the cluster-tilting objects and their mutations in the same cluster category (considered with the $(d+2)$-angulated structure).

We note that all the $d$-exangulated categories appearing in Theorems~\ref{thm:intro_1} and \ref{thm:intro_2} satisfy a higher-dimensional version of defining properties of $0$-Auslander categories from \cite{GNP2} and so can be thought of $0$-Auslander $d$-exangulated categories. A version of this notion has been introduced in \cite{her_n_exang}; see \cite[Remark 4.5]{GNP2} for a discussion on weakenings of the definition in the extriangulated case. 


As mentioned above, all known proofs of the general version of (2) in the extriangulated case make heavy use of dg enhancements.
The Auslander correspondence and its higher analogues due to Iyama have been recently considered in the context of dg-enhanced $(d+2)$-angulated categories by Jasso and Muro and have found important applications in the homological minimal model program as observed in their work with Keller \cite{JKM1, JKM2}. Dg enhancements of more general $d$-exangulated categories have been introduced very recently by Mochizuki and Nakaoka \cite{MN} under the name of higher exact dg categories.
We expect Theorem~\ref{thm:intro_2} to fit into a higher-dimensional version of Chen's 0-Auslander correspondence \cite{chen20230,Chen24}, which would take place in the framework of higher exact dg categories.

\section*{Acknowledgments}

This note was written in 2021--2022 
and was originally intended as a bridge between 
(then unfinished) \cite{FGPPP} and \cite{GNP2}, on one side, and \cite{njw-mut, njw-icra}, on the other side. By the time the last of these works was finished, our focus had 
shifted to other topics. It was not until May 2026 that we looked at this again and realised that we had an essentially complete draft on our shelves. MG would like to thank Xin Fang, Hiroyuki Nakaoka, Yann Palu, Pierre-Guy Plamondon, and Matthew Pressland for the aforementioned collaborations related to this work.

Parts of this work were done during stays of MG at the University of Stuttgart, and he is very grateful to Steffen Koenig for the hospitality. MG acknowledges support by the French ANR grant CHARMS (ANR-19-CE40-0017) and by the Deutsche Forschungsgemeinschaft (DFG, German Research Foundation) – SFB 1624 – ``Higher structures, moduli spaces and integrability'' – 506632645.
NJW is currently supported by EPSRC grant EP/W001780/1.

\section{Background}

We refer to \cite{hln} for definition and basic properties of $d$-exangulated categories. These generalise $(d+2)$-angulated categories of \cite{gko}. For $d = 1$, $d$-exangulated categories are precisely the same as extriangulated categories of Nakaoka--Palu \cite{NP1}. These in turn axiomatise extension-closed subcategories of triangulated categories, and in particular generalise Quillen exact categories and triangulated categories.

All the categories we consider in this paper are additive, linear over a field $K$, $\Hom$-finite, and idempotent-complete. 
In particular, the Krull--Remak--Schmidt property holds for them (for example, \cite{shah}); we rely on this heavily, considering only indecomposable objects in some of the proofs.  

\bigskip\bigskip

We first give background on the higher Auslander algebras of type $A$ \cite{iy-clus} and the structure of relevant subcategories of their module category and derived category, as well as their $(d + 2)$-angulated cluster categories.

\begin{convention}
Our convention will be that, given a tuple $A \in [m]^{d + 1}$, the entries of $A$ are always $(a_{0}, a_{1}, \dots, a_{d})$. The same applies to tuples denoted by other letters.
\end{convention}

\subsubsection{The higher Auslander algebras of type $A$}

We begin by defining the higher Auslander algebras of type~$A$ using quivers with relations. Following \cite{ot}, we denote
\begin{align*}
\modset{m}{d} &:= \{A \in [m]^{d + 1} \st \forall i \in \{0, 1, \dots, d - 1\}, a_{i + 1} \geqslant a_{i} + 2 \}.
\end{align*}

Let $Q^{(d, n)}$ be the quiver with vertices \[Q_{0}^{(d, n)} := \modset{n + 2d - 2}{d - 1}\] and arrows \[Q_{1}^{(d, n)} := \{A \rightarrow A + 1_{i} \st A, A + 1_{i} \in Q_{0}^{(d, n)}\},\] where \[1_{i}:= (0, \dots, 0, \overset{i}{1}, 0, \dots, 0).\] Examples of these quivers are shown in Figure~\ref{fig:quiv_example}.

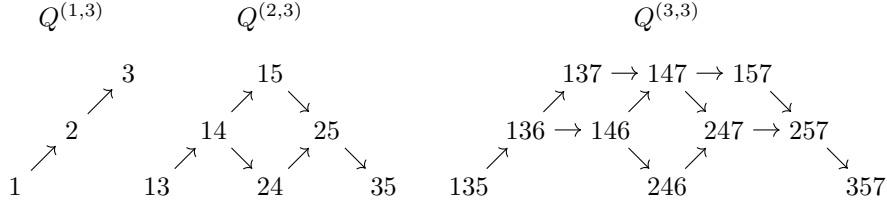
\begin{figure}
\caption{Examples of the quivers $Q^{(d, n)}$}\label{fig:quiv_example}
\[
\begin{tikzpicture}[scale=0.75]

\begin{scope}[shift={(-3,0)}]


\node(1) at (0,0) {1};
\node(2) at (1,1) {2};
\node(3) at (2,2) {3};

\draw[->] (1) -- (2);
\draw[->] (2) -- (3);


\node at (1,3) {$Q^{(1, 3)}$};

\begin{scope}[shift={(-0.5,0)}]


\node(13) at (3,0) {13};
\node(14) at (4,1) {14};
\node(15) at (5,2) {15};
\node(24) at (5,0) {24};
\node(25) at (6,1) {25};
\node(35) at (7,0) {35};

\draw[->] (13) -- (14);
\draw[->] (14) -- (15);
\draw[->] (14) -- (24);
\draw[->] (15) -- (25);
\draw[->] (24) -- (25);
\draw[->] (25) -- (35);


\node at (5,3) {$Q^{(2, 3)}$};

\end{scope}

\end{scope}

\begin{scope}[shift={(-3.5,0)}]


\node(135) at (8.5,0) {135};
\node(136) at (9.5,1) {136};
\node(137) at (10.5,2) {137};
\node(146) at (11,1) {146};
\node(147) at (12,2) {147};
\node(157) at (13.5,2) {157};
\node(246) at (12,0) {246};
\node(247) at (13,1) {247};
\node(257) at (14.5,1) {257};
\node(357) at (15.5,0) {357};

\draw[->] (135) -- (136);
\draw[->] (136) -- (137);
\draw[->] (136) -- (146);
\draw[->] (137) -- (147);
\draw[->] (146) -- (147);
\draw[->] (147) -- (157);
\draw[->] (146) -- (246);
\draw[->] (147) -- (247);
\draw[->] (157) -- (257);
\draw[->] (246) -- (247);
\draw[->] (247) -- (257);
\draw[->] (257) -- (357);


\node at (12,3) {$Q^{(3, 3)}$};

\end{scope}

\end{tikzpicture}
\]
\end{figure}

Let $A_{n}^{d}$ be the quotient of the path algebra $KQ^{(d, n)}$ by the relations:
\begin{equation*}
A \rightarrow A + 1_{i} \rightarrow A + 1_{i} + 1_{j} = 
\left\{
	\begin{array}{ll}
		A \rightarrow A + 1_{j} \rightarrow A + 1_{i} + 1_{j} & \quad \text{if } A + 1_{j} \in Q_{0}^{(d, n)} \\
		\hfil 0 & \quad \text{otherwise.}
	\end{array}
\right.
\end{equation*}
We multiply arrows as if we were composing functions, so that $\xrightarrow{\alpha}\xrightarrow{\beta}=\beta\alpha$.

We will be able to describe many aspects of the representation theory of these algebras combinatorially. In particular, the following notion will be key. We say that a tuple $A$ \emph{intertwines} a tuple $B$ if \[a_{0} < b_{0} < a_{1} < b_{1} < \dots < a_{d} < b_{d}.\] In this case, we write $A \wr B$. If $A \wr B$ or $B \wr A$, then we say that $A$ and $B$ \emph{are intertwining} and write $A \swr B$.

\subsubsection{The module category}\label{sect:hart:hara:modcat}

Given a finite-dimensional algebra $\Lambda$ over a field $K$, let $\mathcal{M}$ be a functorially finite subcategory of $\modules \Lambda$.
(For the notion of functorially finite, see for instance \cite[Definition~2.1]{iya-mos}.)
Then we call $\mathcal{M}$ \emph{$d$-cluster-tilting} if
\begin{align*}
\mathcal{M} &= \{ X \in \modules\Lambda \st \forall i \in [d-1], \forall M \in \mathcal{M}, \mathrm{Ext}_{\Lambda}^{i}(X,M) = 0 \} \\
&= \{ X \in \modules\Lambda \st \forall i \in [d-1], \forall M \in \mathcal{M}, \mathrm{Ext}_{\Lambda}^{i}(M,X) = 0\}.
\end{align*}

In the case $d = 1$, the conditions should be interpreted as being trivial, so that $\modules \Lambda$ is the unique $1$-cluster-tilting subcategory of $\modules \Lambda$. If $\additive M$ is a $d$-cluster-tilting subcategory, for $M \in \modules \Lambda$, then we say that $M$ is a \emph{$d$-cluster-tilting module}. We say that $\Lambda$ is \emph{$d$-representation-finite $d$-hereditary} if there exists a $d$-cluster-tilting module in $\modules \Lambda$ and $\mathrm{gl.dim}\,\Lambda \leqslant d$, following \cite{io,hio,jk-nak}. It is shown in \cite{iy-clus} that the algebra $A_{n}^{d}$ is $d$-representation-finite $d$-hereditary with unique basic $d$-cluster-tilting module $M^{(d,n)}$ and that \[ A_{n}^{d+1} \cong \End_{A_{n}^{d}}M^{(d,n)}.\]

\begin{notation}
We use the following notation from \cite{ot}: for $A, B \in \mathbb{Z}^{d + 1}$ and $I \subseteq \{0, 1, \dots, d\}$, we write $m_{I}(A, B) = C$, where $c_{i} = a_{i}$ if $i \in I$ and $c_{i} = b_{i}$ if $i \notin I$.

We also write $\onevec$ for the tuple $(1, 1, \dots, 1)$.
\end{notation}

Oppermann and Thomas show that the category $\additive M^{(d, n)}$ may be described combinatorially. 

\begin{theorem}[{\cite[Theorem~3.6, Theorem~3.8(4), Proposition~3.12(3)]{ot}}]\label{thm:mod_desc}
There is a bijection $A \mapsto M_{A}$ between $\modset{n + 2d}{d}$ and the isomorphism classes of indecomposable modules of $\additive M^{(d, n)}$ such that the following properties hold.
\begin{enumerate}
\item $M_{A}$ is projective if and only if $a_{0} = 1$.\label{op:mod_desc:proj}
\item $M_{A}$ is injective if and only if $a_{d} = n + 2d$.\label{op:mod_desc:inj}
\item $\Hom_{A_{n}^{d}}(M_{B}, M_{A}) \neq 0$ if and only if $(B - \onevec) \wr A$, and in this case the $\Hom$-space is one-dimensional.\label{op:mod_desc:hom}
\item $\Ext_{A_{n}^{d}}^{d}(M_{B}, M_{A}) \neq 0$ if and only if $A \wr B$, and in this case the $\Ext$-space is one-dimensional.\label{op:mod_desc:ext}
\item Moreover, if $\Ext_{A_{n}^{d}}^{d}(M_{B}, M_{A}) \neq 0$, then there is a non-split exact sequence \[0 \to M_{A} \to E_{d} \to E_{d - 1} \to \dots \to E_{1} \to M_{B} \to 0,\] where \[E_{r} = \bigoplus_{\substack{I \subseteq \{0, 1, \dots, d\} \\ m_{I}(A, B) \in \modset{n + 2d}{d} \\ |I| = r}} M_{m_{I}(A, B)}.\]\label{op:mod_desc:seq}
\item The unique-up-to-scalar morphisms $M_{A} \to M_{B}$ and $M_{B} \to M_{C}$ compose to give a non-zero morphism $M_{A} \to M_{C}$ if and only if $(A - \onevec) \wr C$.\label{op:mod_desc:comp}
\end{enumerate}
\end{theorem}

The category $\additive M^{(d, n)}$ is $d$-exact in the sense of \cite{jasso}, and therefore $d$-exangulated by \cite[Proposition 4.34]{hln}. Indeed, we have that $\bbe(M_{B}, M_{A}) = \Ext_{A_{n}^{d}}^{d}(M_{B}, M_{A})$, with the exact realisation $\fs$ given by sending elements of $\bbe(M_{B}, M_{A})$ to homotopy equivalence classes of the sequence in (\ref{op:mod_desc:seq}), provided this extension space is non-zero. This exact realisation is unique up to scalar. The category $\additive M^{(2, 3)}$ is illustrated in Figure~\ref{fig:A32_ar_quiv}.

\begin{figure}
\caption{The category $\additive M^{(2, 3)}$}\label{fig:A32_ar_quiv}
\[
\begin{tikzpicture}

\node(135) at (-3.5,0) {$M_{135}$};
\node(136) at (-2.5,1) {$M_{136}$};
\node(137) at (-1.5,2) {$M_{137}$};
\node(146) at (-1,1) {$M_{146}$};
\node(147) at (0,2) {$M_{147}$};
\node(157) at (1.5,2) {$M_{157}$};
\node(246) at (0,0) {$M_{246}$};
\node(247) at (1,1) {$M_{247}$};
\node(257) at (2.5,1) {$M_{257}$};
\node(357) at (3.5,0) {$M_{357}$};

\draw[->] (135) -- (136);
\draw[->] (136) -- (137);
\draw[->] (136) -- (146);
\draw[->] (137) -- (147);
\draw[->] (146) -- (147);
\draw[->] (147) -- (157);
\draw[->] (146) -- (246);
\draw[->] (147) -- (247);
\draw[->] (157) -- (257);
\draw[->] (246) -- (247);
\draw[->] (247) -- (257);
\draw[->] (257) -- (357);

\end{tikzpicture}
\]
\end{figure}
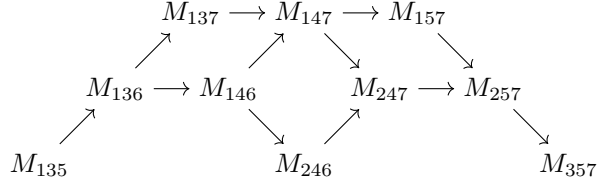

\subsubsection{The derived category}\label{sect:hart:hara:dercat}

Given a triangulated category $\mathcal{D}$, a functorially finite subcategory $\mathcal{C}$ of $\mathcal{D}$ is called \emph{$d$-cluster-tilting} if
\begin{align*}
\mathcal{C} &= \{X \in \mathcal{D} \st \forall i \in [d - 1], \forall Y \in \mathcal{C}, \Hom_{\mathcal{D}}(X, Y[i]) = 0\}\\
&= \{X \in \mathcal{D} \st \forall i \in [d - 1], \forall Y \in \mathcal{C}, \Hom_{\mathcal{D}}(Y, X[i]) = 0\}.
\end{align*}
Gei\ss, Keller, and Oppermann show that a $d$-cluster-tilting subcategory of a triangulated category is a \emph{$(d + 2)$-angulated category} \cite{gko}.

Given a $d$-representation-finite $d$-hereditary algebra with $d$-cluster-tilting module~$M$, the subcategory \[\derclus := \additive \{M[id] \in \mathcal{D}_{\Lambda} \st i \in \mathbb{Z}\}\] is a $d$-cluster-tilting subcategory of $\mathcal{D}_{\Lambda}$. Here we have used $\mathcal{D}_{\Lambda}$ to denote the bounded derived category of finitely generated $\Lambda$-modules $\bder$.

Just as the indecomposable objects of $M^{(d, n)}$ may be labelled by tuples in a way that concords with homomorphisms and extensions, the indecomposables of $\derclusa$ also admit a similar labelling. We denote \[\derset{m}{d} = \left\{(a_{0},\dots,a_{d}) \in \mathbb{Z}^{d+1} \,\middle|\, \parbox{5cm}{\begin{center}$\forall i \in \{ 0, 1, \dots ,d-1 \},$ $a_{i+1}\geqslant a_{i}+2 \text{ and } a_{d} + 2 \leqslant a_{0} + m$\end{center}}\right\}.\] We have the following result.

\begin{theorem}[{\cite[Proof of Proposition 6.1, Lemma 6.6, Lemma 6.7]{ot}}]\label{thm:der_desc}
There is a bijection $A \mapsto \U_{A}$ between $\derset{n + 2d + 1}{d}$ and the isomorphism classes of indecomposable objects of $\derclusa$ such that the following properties hold.
\begin{enumerate}
\item $\U_{A}[d] = \U_{(a_{1} - 1, a_{2} - 1, \dots, a_{d} - 1, a_{0} + n + 2d)}$.\label{op:der_desc:shift}
\item $\Hom_{\dla}(\U_{B}, \U_{A}) \neq 0$ if and only if \[b_{0} - 1 < a_{0} < b_{1} - 1 < a_{1} < \dots < b_{d} - 1 < a_{d} < b_{0} + n + 2d.\] and in this case the $\Hom$-space is one-dimensional.\label{op:der_desc:hom}
\item $\Hom_{\dla}(\U_{B}, \U_{A}[d]) \neq 0$ if and only if \[a_{0} < b_{0} < a_{1} < b_{1} < \dots < a_{d} < b_{d} < a_{0} + n + 2d + 1,\] with this space one-dimensional.\label{op:der_desc:ext}
\item Moreover, if $\Hom_{\dla}(\U_{B}, \U_{A}[d]) \neq 0$, then there is a $(d + 2)$-angle \[\U_{A} \to F_{d} \to F_{d - 1} \to \dots \to F_{1} \to \U_{B} \to \U_{A}[d],\] where \[F_{r} = \bigoplus_{\substack{I \subseteq \{0, 1, \dots, d\} \\ m_{I}(A, B) \in \derset{n + 2d + 1}{d} \\ |I| = r}}  \U_{m_{I}(A, B)},\] which is not nullhomotopic.\label{op:der_desc:seq}
\item The unique-up-to-scalar morphisms $\U_{A} \to \U_{B}$ and $\U_{B} \to \U_{C}$ compose to give a non-zero morphism $\U_{A} \to \U_{C}$ if and only if \[a_{0} - 1 < c_{0} < a_{1} - 1 < c_{1} < \dots < a_{d} - 1 < c_{d} < a_{0} + n + 2d.\]\label{op:der_desc:comp}
\end{enumerate}
\end{theorem}

By \cite[Theorem 1]{gko}, we have that $\derclusa$ is a $(d + 2)$-angulated category. Hence, by \cite[Proposition 4.5]{hln}, we have that $\derclusa$ is a $d$-exangulated category. We have that $\bbe(\U_{B}, \U_{A}) = \Hom_{\dla}(\U_{B}, \U_{A}[d])$, with the exact realisation $\fs$ given by sending elements of $\bbe(\U_{B}, \U_{A})$ to homotopy equivalence classes of the $(d + 2)$-angle in (\ref{op:mod_desc:seq}), provided this extension space is non-zero. This exact realisation is unique up to scalar. The category $\argclus{A_{3}^{2}}$ is illustrated in Figure~\ref{fig:a32_dercat_lab}.

\begin{figure}
\caption{The category $\argclus{A_{3}^{2}}$}\label{fig:a32_dercat_lab}
\[
\begin{tikzpicture}[xscale=0.7]


\node(zzcz) at (0,0){$\U_{135}$};
\node(azbz) at (6,0){$\U_{146}$};
\node(bzaz) at (12,0){$\U_{157}$};
\node(zabz) at (3,1){$\U_{136}$};
\node(aaaz) at (9,1){$\U_{147}$};
\node(zbaz) at (6,2){$\U_{137}$};

\node(zzca) at (2,3){$\U_{246}$};
\node(azba) at (8,3){$\U_{257}$};
\node(bzaa) at (14,3){$\U_{268}$};
\node(zaba) at (5,4){$\U_{247}$};
\node(aaaa) at (11,4){$\U_{258}$};
\node(zbaa) at (8,5){$\U_{248}$};

\node(zzcb) at (4,6){$\U_{357}$};
\node(azbb) at (10,6){$\U_{368}$};
\node(bzab) at (16,6){$\U_{379}$};
\node(zabb) at (7,7){$\U_{358}$};
\node(aaab) at (13,7){$\U_{369}$};
\node(zbab) at (10,8){$\U_{359}$};

\node at (0,-1){\dots};
\node at (6,-1){\dots};
\node at (12,-1){\dots};

\node at (4,9){\dots};
\node at (10,9){\dots};
\node at (16,9){\dots};


\draw[->] (zzcz) -- (zabz);
\draw[->] (zabz) -- (azbz);
\draw[->] (azbz) -- (aaaz);
\draw[->] (aaaz) -- (bzaz);
\draw[->] (zabz) -- (zbaz);
\draw[->] (zbaz) -- (aaaz);

\draw[->] (zzca) -- (zaba);
\draw[->] (zaba) -- (azba);
\draw[->] (azba) -- (aaaa);
\draw[->] (aaaa) -- (bzaa);
\draw[->] (zaba) -- (zbaa);
\draw[->] (zbaa) -- (aaaa);

\draw[->] (zzcb) -- (zabb);
\draw[->] (zabb) -- (azbb);
\draw[->] (azbb) -- (aaab);
\draw[->] (aaab) -- (bzab);
\draw[->] (zabb) -- (zbab);
\draw[->] (zbab) -- (aaab);

\draw[->] (azbz) -- (zzca);
\draw[->] (bzaz) -- (azba);
\draw[->] (aaaz) -- (zaba);

\draw[->] (azba) -- (zzcb);
\draw[->] (bzaa) -- (azbb);
\draw[->] (aaaa) -- (zabb);

\end{tikzpicture}
\]
\end{figure}

\subsubsection{The cluster category}

Given a $d$-representation-finite $d$-hereditary algebra $\Lambda$, the \emph{$(d + 2)$-angulated cluster category} of $\Lambda$ \cite[Definition 5.22]{ot} is defined to be the orbit category \[\otclus = \frac{\derclus}{\nu_{d}[-d]}\,,\] where $\derclus$ is the $d$-cluster-tilting subcategory of $\mathcal{D}_{\Lambda}$ from the previous section.

We can describe the $(d + 2)$-angulated cluster category $\otclusa$ combinatorially using the set \[\nonconsec{m}{d} := \{A \in \modset{m}{d} \st a_{d} \leqslant a_{0} + m - 2\}.\]

\begin{convention}
When we describe the cluster category $\otclusa$ combinatorially, we will use arithmetic modulo $n + 2d + 1$ and implicitly re-order tuples so that they are increasing. This in particular applies when we write things like $A - \onevec$.
\end{convention}

\begin{theorem}[{\cite[Proposition~6.1 and Theorem~5.2(3)]{ot}}]\label{thm:clus_desc}
There is a bijection $A \mapsto O_{A}$ between $\nonconsec{n + 2d + 1}{d}$ and the isomorphism classes of indecomposable objects of $\otclusa$ such that the following properties hold.
\begin{enumerate}
\item $O_{A}[d] = O_{A - \onevec}$.\label{op:clus_desc:shift}
\item $\Hom_{\otclusa}(O_{B}, O_{A}) \neq 0$ if and only if $(B - \onevec) \swr A$\label{op:clus_desc:hom}
\item For indecomposables $O_{A}, O_{B}$ of $\otclusa$, we have that $\Hom_{\otclusa}(O_{B}, O_{A}[d]) \neq 0$ if and only if $A \swr B$.\label{op:clus_desc:ext}
\item Moreover, if $\Hom_{\otclusa}(O_{B}, O_{A}[d]) \neq 0$, then there is a distinguished $(d + 2)$-angle \[O_{A} \to G_{d} \to G_{d - 1} \to \dots \to G_{1} \to O_{B} \to O_{A}[d],\] where \[G_{r} = \bigoplus_{\substack{I \subseteq \{0, 1, \dots, d\} \\ m_{I}(A, B) \in \derset{n + 2d + 1}{d} \\ |I| = r}}  O_{m_{I}(A, B)},\] which is not nullhomotopic.\label{op:clus_desc:seq}
\end{enumerate}
\end{theorem}

Composition of morphisms in the cluster category is slightly more complicated than in the previous cases. We require the following lemma to describe it. Recall that a \emph{cyclically shifted order} on $[m]$ is an order \[l < l+1 < \dots < m-1 < m < 1 < \dots < l-1,\] where $l \in [m]$.
The following lemma is the higher analogue of \cite[Lemma~4.2]{njw-om}.

\begin{lemma}\label{lem:clus_comp}
Suppose that we have non-zero morphisms $O_{A} \to O_{B}$ and $O_{B} \to O_{C}$. Then these morphisms compose to give a non-zero morphism $O_{A} \to O_{C}$ if and only if there is a cyclically shifted ordering of $[n + 2d + 1]$ such that \[a_{0} - 1 \leqslant b_{0} - 1 < c_{0} < a_{1} - 1 \leqslant b_{1} - 1 < c_{1} < \dots < a_{d} - 1 \leqslant b_{d} - 1 < c_{d} < a_{0} + n + 2d.\]
\end{lemma}
\begin{proof}
Suppose that there is a cyclically shifted ordering of $[n + 2d + 1]$ such that \[a_{0} - 1 \leqslant b_{0} - 1 < c_{0} < a_{1} - 1 \leqslant b_{1} - 1 < c_{1} < \dots < a_{d} - 1 \leqslant b_{d} - 1 < c_{d} < a_{0} + n + 2d.\] Then, by applying the automorphism $[d]$ repeatedly, we may assume that $a_{0} = 1$. Hence, our morphisms $O_{A} \to O_{B}$ and $O_{B} \to O_{C}$ then lift to morphisms $\U_{A} \to \U_{B}$ and $\U_{B} \to \U_{C}$ in $\derclusa$ by Theorem~\ref{thm:der_desc}\eqref{op:der_desc:hom}. We then have by Theorem~\ref{thm:der_desc}\eqref{op:der_desc:comp} that these morphisms compose to give a morphism $\U_{A} \to \U_{C}$. The functoriality of the projection from $\derclusa$ to $\otclusa$ then guarantees that our morphisms $O_{A} \to O_{B}$ and $O_{B} \to O_{C}$ compose to give a morphism $O_{A} \to O_{C}$.

Now suppose that $O_{A} \to O_{B}$ and $O_{B} \to O_{C}$ compose to give a non-zero morphism $O_{A} \to O_{C}$. Hence, we have that $(A - \onevec) \swr C$. Choose a cyclically shifted ordering of $[n + 2d + 1]$ such that $(A - \onevec) \wr C$. Again, we may assume that $a_{0} = 1$ by applying the automorphism $[d]$. We then have that the morphism $O_{A} \to O_{C}$ lifts to a morphism $\U_{A} \to \U_{C}$ in $\derclusa$. If we have that the morphisms $\U_{A} \to \U_{B}$ and $\U_{B} \to \U_{C}$ compose to give zero, then this must be true of the images of these morphisms in $\otclusa$. Hence we must have that $\U_{A} \to \U_{B}$ and $\U_{B} \to \U_{C}$ compose to give a non-zero morphism $\U_{A} \to \U_{C}$. By Theorem~\ref{thm:der_desc}\eqref{op:der_desc:hom}, we have that \[a_{0} - 1 \leqslant b_{0} - 1 < c_{0} < a_{1} - 1 \leqslant b_{1} - 1 < c_{1} < \dots < a_{d} - 1 \leqslant b_{d} - 1 < c_{d} < a_{0} + n + 2d.\] This completes the proof.
\end{proof}

We give an example where there are morphisms $O_{A} \to O_{B}$ and $O_{B} \to O_{C}$ which do not compose to give a morphism $O_{A} \to O_{C}$, even though such a morphism exists.

\begin{example}
Consider the following example from $\argot{A_{6}^{1}}$. We have morphisms \[
\begin{tikzcd}
O_{15} \ar[rr] \ar[dr] && O_{26} \\
& O_{48}, \ar[ur] &
\end{tikzcd}
\] as can be verified using Theorem~\ref{thm:clus_desc}\eqref{op:clus_desc:hom}. However, this diagram does not commute, since the morphisms $O_{15} \to O_{26}$ and $O_{15} \to O_{48}$ are respectively images of morphisms $U_{15} \to U_{26}$ and $U_{15} \to U_{48}$, whereas $O_{48} \to O_{26}$ is the image of a morphism $U_{48} \to U_{6,11}$. It can be verified that this case does not satisfy the conditions of Lemma~\ref{lem:clus_comp}.
\end{example}

The category $\argot{A_{3}^{2}}$ is illustrated in Figure~\ref{fig:a32_cluscat_lab}.

\begin{figure}
\caption{The category $\argot{A_{3}^{2}}$}\label{fig:a32_cluscat_lab}
\[
\begin{tikzpicture}[xscale=0.7]


\node(zzcz) at (0,0){$O_{135}$};
\node(azbz) at (6,0){$O_{146}$};
\node(bzaz) at (12,0){$O_{157}$};
\node(zabz) at (3,1){$O_{136}$};
\node(aaaz) at (9,1){$O_{147}$};
\node(zbaz) at (6,2){$O_{137}$};

\node(zzca) at (2,3){$O_{246}$};
\node(azba) at (8,3){$O_{257}$};
\node(bzaa) at (14,3){$O_{268}$};
\node(zaba) at (5,4){$O_{247}$};
\node(aaaa) at (11,4){$O_{258}$};
\node(zbaa) at (8,5){$O_{248}$};

\node(zzcb) at (4,6){$O_{357}$};
\node(azbb) at (10,6){$O_{368}$};
\node(bzab) at (16,6){$O_{137}$};
\node(zabb) at (7,7){$O_{358}$};
\node(aaab) at (13,7){$O_{136}$};
\node(zbab) at (10,8){$O_{135}$};

\node at (0,-1){\dots};
\node at (6,-1){\dots};
\node at (12,-1){\dots};

\node at (4,9){\dots};
\node at (10,9){\dots};
\node at (16,9){\dots};


\draw[->] (zzcz) -- (zabz);
\draw[->] (zabz) -- (azbz);
\draw[->] (azbz) -- (aaaz);
\draw[->] (aaaz) -- (bzaz);
\draw[->] (zabz) -- (zbaz);
\draw[->] (zbaz) -- (aaaz);

\draw[->] (zzca) -- (zaba);
\draw[->] (zaba) -- (azba);
\draw[->] (azba) -- (aaaa);
\draw[->] (aaaa) -- (bzaa);
\draw[->] (zaba) -- (zbaa);
\draw[->] (zbaa) -- (aaaa);

\draw[->] (zzcb) -- (zabb);
\draw[->] (zabb) -- (azbb);
\draw[->] (azbb) -- (aaab);
\draw[->] (aaab) -- (bzab);
\draw[->] (zabb) -- (zbab);
\draw[->] (zbab) -- (aaab);

\draw[->] (azbz) -- (zzca);
\draw[->] (bzaz) -- (azba);
\draw[->] (aaaz) -- (zaba);

\draw[->] (azba) -- (zzcb);
\draw[->] (bzaa) -- (azbb);
\draw[->] (aaaa) -- (zabb);

\end{tikzpicture}
\]
\end{figure}


\subsubsection{The $d$-almost positive category}

Given a $d$-representation-finite $d$-hereditary algebra $\Lambda$ with $d$-cluster-tilting module $M$, define the \emph{$d$-almost positive} subcategory $\dsubcat$ to be $\additive (M \oplus \Lambda[d])$ of $\derclus$, following \cite{njw-icra}.

\begin{remark}
For $d = 1$, the almost positive category is the category of two-term complexes of projectives of $\Lambda$. For $d > 1$, we do not refer to the $d$-almost positive category as the category of $d$-term complexes of projectives, since it does not contain all complexes of projectives with $d$ terms. The name `$d$-almost positive category' is chosen due to the fact that for a hereditary algebra of Dynkin type, the indecomposable objects of the almost positive category are in bijection with the almost positive roots of the root system. In \cite{OT26}, this category was discussed under the name of a ``fundamental domain''; we prefer to avoid using this name  since for $d = 1$ this category does not coincide with another ``fundamental domain'' in the cluster related literature: the additive subcategory of the perfect derived category of the Ginzburg dg algebra as first considered by Amiot \cite{amiot}.
\end{remark}

We can describe the $d$-almost positive category of $A_{n}^{d}$ combinatorially using the set $\nonconsec{m}{d}$ again. We have that $\derclusa$ is a $(d + 2)$-angulated category and that $\dsubcata$ is a $d$-extension closed subcategory of it. Hence, by \cite[Proposition 2.35]{hln}, we have that $\dsubcata$ is a $d$-exangulated category with $\bbe(\U_{B}, \U_{A}) = \Hom_{\derclusa}(\U_{B}, \U_{A}[d])$.

\begin{theorem}\label{thm:dcomp_desc}
There is a bijection $A \mapsto \U_{A}$ between $\nonconsec{n + 2d + 1}{d}$ and the isomorphism classes of indecomposable objects of $\dsubcata$ such that the following properties hold.
\begin{enumerate}
\item $\Hom_{\dsubcata}(\U_{B}, \U_{A}) \neq 0$ if and only if \[b_{0} - 1 < a_{0} < b_{1} - 1 < a_{1} < \dots < b_{d} - 1 < a_{d} < b_{0} + n + 2d.\]\label{op:dcomp_desc:hom}
\item For indecomposables $\U_{A}, \U_{B}$ of $\dsubcata$, we have $\bbe(\U_{B}, \U_{A}) \neq 0$ if and only if $B \wr A$.\label{op:dcomp_desc:ext}
\item Moreover, if $\bbe(\U_{B}, \U_{A}) \neq 0$, then the exact realisation $\fs$ is given by scalar multiples of the homotopy equivalence class of the distinguished $d$-exangle \[\U_{A} \to H_{d} \to H_{d - 1} \to \dots \to H_{1} \to \U_{B} \dashrightarrow,\] where \[H_{r} = \bigoplus_{\substack{I \subseteq \{0, 1, \dots, d\} \\ m_{I}(A, B) \in \derset{n + 2d + 1}{d} \\ |I| = r}}  \U_{m_{I}(A, B)},\] which is not nullhomotopic.\label{op:dcomp_desc:seq}
\item The unique-up-to-scalar morphisms $\U_{A} \to \U_{B}$ and $\U_{B} \to \U_{C}$ compose to give a non-zero morphism $\U_{A} \to \U_{C}$ if and only if \[a_{0} - 1 < c_{0} < a_{1} - 1 < c_{1} < \dots < a_{d} - 1 < c_{d} < a_{0} + n + 2d.\]\label{op:dcomp_desc:comp}
\end{enumerate}
\end{theorem}
\begin{proof}
We have that the indecomposable objects of $\additive M^{(d, n)}$ are in bijection with $\modset{n + 2d}{d}$ by Theorem~\ref{thm:mod_desc}. By Theorem~\ref{thm:mod_desc}\eqref{op:mod_desc:proj}, a module $M_{A}$ is projective if and only if $a_{0} = 1$. It is then straightforward to see from Theorem~\ref{thm:der_desc}\eqref{op:der_desc:shift} that an object $\U_{A}$ in $\derclusa$ is equal to $M_{A}[d]$ for $M_{A}$ projective if and only if $A \in \nonconsec{n + 2d + 1}{d}$ with $a_{d} = n + 2d + 1$. Since \[\nonconsec{n + 2d + 1}{d} = \modset{n + 2d}{d} \sqcup \{A \in \nonconsec{n + 2d + 1}{d} \st a_{d} = n + 2d + 1\},\] the result follows. We then obtain the other statements in the following ways.
\begin{enumerate}
\item This follows immediately from Theorem~\ref{thm:der_desc}\eqref{op:der_desc:hom}.

\item It follows from Theorem~\ref{thm:der_desc}\eqref{op:der_desc:ext} that that $\Hom_{\derclusa}(\U_{B}, \U_{A}[d]) \neq 0$ if and only if \[a_{0} < b_{0} < a_{1} < b_{1} < \dots < a_{d} < b_{d} < a_{0} + n + 2d + 1.\] Since the indecomposable objects are labelled by $\nonconsec{n + 2d + 1}{d}$, we have that $a_{0} \geqslant 1$ and $b_{d} \leqslant n + 2d + 1$, so it is automatic that $b_{d} < a_{0} + n + 2d + 1$. Hence, we obtain the result.

\item This follows immediately from Theorem~\ref{thm:der_desc}\eqref{op:der_desc:seq}.

\item This also follows immediately from Theorem~\ref{thm:der_desc}\eqref{op:der_desc:comp}.
\end{enumerate}
\end{proof}

\begin{figure}
\caption{The category $\argds{A_{3}^{2}}$}\label{fig:a32_dap_lab}
\[
\begin{tikzpicture}[xscale=0.8]

\node(135) at (-3.5,0) {$\U_{135}$};
\node(136) at (-2.5,1) {$\U_{136}$};
\node(137) at (-1.5,2) {$\U_{137}$};
\node(146) at (-1,1) {$\U_{146}$};
\node(147) at (0,2) {$\U_{147}$};
\node(157) at (1.5,2) {$\U_{157}$};
\node(246) at (1,0) {$\U_{246}$};
\node(247) at (2,1) {$\U_{247}$};
\node(257) at (3.5,1) {$\U_{257}$};
\node(248) at (3,2) {$\U_{248}$};
\node(258) at (4.5,2) {$\U_{258}$};
\node(268) at (6,2) {$\U_{268}$};
\node(357) at (5.5,0) {$\U_{357}$};
\node(358) at (6.5,1) {$\U_{358}$};
\node(368) at (8,1) {$\U_{368}$};
\node(468) at (10,0) {$\U_{468}$};

\draw[->] (135) -- (136);
\draw[->] (136) -- (137);
\draw[->] (136) -- (146);
\draw[->] (137) -- (147);
\draw[->] (146) -- (147);
\draw[->] (147) -- (157);
\draw[->] (146) -- (246);
\draw[->] (147) -- (247);
\draw[->] (157) -- (257);
\draw[->] (246) -- (247);
\draw[->] (247) -- (257);
\draw[->] (257) -- (357);
\draw[->] (247) -- (248);
\draw[->] (248) -- (258);
\draw[->] (258) -- (268);
\draw[->] (257) -- (258);
\draw[->] (357) -- (358);
\draw[->] (258) -- (358);
\draw[->] (268) -- (368);
\draw[->] (368) -- (468);

\end{tikzpicture}
\]
\end{figure}

\section{Relation between the module category and the \texorpdfstring{$d$}{d}-almost positive category}

We can now state and prove the relation between $\modules A_{n + 1}^{d}$ and $\dsubcata$. We write $\mathcal{P}$ for the category of projective $A_{n + 1}^{d}$-modules and $\mathcal{I}$ for the category of injective $A_{n + 1}^{d}$-modules, so that $\projinj$ is the category of projective-injective modules.

\begin{theorem}\label{thm:equiv}
We have that $\quotcat := \additive (M^{(d, n + 1)})/\projinj$ is equivalent to $\shortap := \dsubcata$ as a $d$-exangulated category.
\end{theorem}

For extriangulated categories, the quotient by a subcategory of the projective-injectives is an extriangulated category with the distinguished extriangles of the quotient category given by the images of the distinguished extriangles of the original category under the quotient. However, taking the images of the distinguished $d$-exangles under a quotient by a subcategory of the projective-injectives does not always give a well-defined $d$-exangulated category, by \cite[Example 3.3]{hzz}. Hence, Theorem~\ref{thm:equiv} also shows that $\quotcat$ is a well-defined $d$-exangulated category under this natural set of distinguished $d$-exangles, which is not immediate.

\begin{proof}[Proof of Theorem~\ref{thm:equiv}]
We prove the result using the combinatorial interpretation of the categories from the previous section. The equivalence between $\quotcat$ and $\shortap$ is defined on the indecomposables by $M_{A} \mapsto \U_{A}$. For this to give an equivalence of $d$-exangulated categories, we need it to give an equivalence of categories which induces isomorphisms between extension spaces such that corresponding elements of the extension spaces are realised by the same $d$-exangles.

By Theorem~\ref{thm:dcomp_desc}, we have that the indecomposables of $\shortap$ are in bijection with $\nonconsec{n + 2d + 1}{d}$. By Theorem~\ref{thm:mod_desc}, we have that the indecomposables of $\additive M^{(d, n + 1)}$ are in bijection with $\modset{n + 2d + 1}{d}$. The projective-injective indecomposables correspond to \[\{A \in \modset{n + 2d + 1}{d} \st a_{0} = 1, a_{d} = n + 2d + 1\}\] by Theorem~\ref{thm:mod_desc}(\ref{op:mod_desc:proj}) and (\ref{op:mod_desc:inj}). It is then easy to see that \[\nonconsec{n + 2d + 1}{d} = \modset{n + 2d + 1}{d} \setminus \{A \in \modset{n + 2d + 1}{d} \st a_{0} = 1, a_{d} = n + 2d + 1\}.\] Hence we have a bijection between isomorphism classes of indecomposable objects of $\shortap$ and isomorphism classes of indecomposable objects in $\quotcat$.

We now consider homomorphisms between indecomposable objects in the two categories. We already know from Theorem~\ref{thm:dcomp_desc}(\ref{op:dcomp_desc:hom}) that $\Hom_{\shortap}(\U_{B}, \U_{A}) \neq 0$ if and only if \[b_{0} - 1 < a_{0} < b_{1} - 1 < a_{1} < \dots < b_{d} - 1 < a_{d} < b_{0} + n + 2d.\] On the other hand, we have that $\Hom_{\quotcat}(M_{B}, M_{A}) \neq 0$ if and only if $\Hom_{A_{n}^{d}}(M_{B},$ $M_{A})\neq 0$ and the unique-up-to-scalar homomorphism does not factor through a projective-injective. From Theorem~\ref{thm:mod_desc}(\ref{op:mod_desc:hom}), we have that $\Hom_{A_{n}^{d}}(M_{B}, M_{A}) \neq 0$ if and only if \[b_{0} - 1 < a_{0} < b_{1} - 1 < a_{1} < \dots < b_{d} - 1 < a_{d}.\] Applying this result, we see that there is a projective-injective $M_{C}$ such that there is a factorisation $M_{B} \to M_{C} \to M_{A}$ if and only if we have $b_{0} = 1$ and $a_{d} = n + 2d + 1$, since $c_{0} = 1$ and $c_{d} = n + 2d + 1$. Indeed, if $b_{0} = 1$ and $a_{d} = n + 2d + 1$, then we always have factorisations $M_{B} \to M_{(1, b_{1}, \dots, b_{d - 1}, n + 2d + 1)} \to M_{A}$ and $M_{B} \to M_{(1, a_{1}, \dots, a_{d - 1}, n + 2d + 1)} \to M_{A}$. Hence there exists no factorisation through a projective-injective if and only if $a_{d} < b_{0} + n + 2d$. We therefore obtain that $\Hom_{\quotcat}(M_{B}, M_{A}) \neq 0$ if and only if \[b_{0} - 1 < a_{0} < b_{1} - 1 < a_{1} < \dots < b_{d} - 1 < a_{d} < b_{0} + n + 2d,\] as desired.

This furthermore allows us to show that the map $M_{A} \mapsto \U_{A}$ defines a functor. Indeed, we have that $\Hom_{\quotcat}(M_{B}, M_{A})$ $\neq 0$ if and only if $\Hom_{\shortap}(\U_{B}, \U_{A})$ $\neq 0$. Since the $\Hom$-spaces are both one-dimensional in this case, we therefore get an isomorphism of $\Hom$-spaces, which defines the map that the functor induces on morphisms $\Hom_{\quotcat}(M_{B}, M_{A}) \xrightarrow{\sim} \Hom_{\shortap}(\U_{B}, \U_{A})$. It follows from \cite[Proposition~3.12(3)]{ot} that these isomorphisms of $\Hom$-spaces can be chosen in a way that is compatible with composition of morphisms.

We finally consider extensions between indecomposable objects in the two categories. We have that $\bbe_{\mathcal{B}}(M_{A}, M_{B}) \neq 0$ if and only if $B \wr A$ if and only if $\bbe_{\mathcal{C}}(\U_{A}, \U_{B}) \neq 0$, as desired. Moreover, Theorem~\ref{thm:mod_desc}\eqref{op:mod_desc:seq} and Theorem~\ref{thm:der_desc}\eqref{op:dcomp_desc:seq} show that these extensions are realised by the same $d$-exangles via their exact realisations~$\fs$.
\end{proof}

This result can be used to explain the fact that tilting modules for $A_{n + 1}$ are in bijection with support tilting modules for $A_{n}$, as observed in \cite[N 4.8]{ringel_catalan}.
The higher-dimensional version of this is that tilting modules for $A_{n + 1}^{d}$ in $\additive M^{(d, n + 1)}$ are in bijection with silting objects for $A_{n}^{d}$ in $\dsubcata$ \cite{ot,njw-hst,njw-icra}.
The last paragraph of the proof Theorem~\ref{thm:equiv} further ensures that the quotient functor sends the exchange $d$-exangles for mutation of tilting modules for $A^d_{n+1}$ in $\additive M^{(d, n + 1)}$ precisely to the exchange $d$-exangles for mutation of silting objects for $A^d_{n+1}$ in $\dsubcata = \mathcal{C}$, and so the respective mutations are intertwined by the quotient functor.

\section{Relation between the cluster category and the \texorpdfstring{$d$}{d}-almost positive category}

We now show the relation between the cluster category and the $d$-almost positive category in terms of $d$-exangulated structures. We first consider a different $d$-exangulated structure on $\otclusa$ where the distinguished $d$-exangles are given by distinguished $(d + 2)$-angles \[O_{1} \to G_{d} \to G_{d - 1} \to \dots \to G_{1} \to O_{2} \to O_{1}[d]\] where the morphism $O_{2} \to O_{1}[d]$ factors through $O_{3}[d]$, where $O_{3}$ is the image of a projective $A_{n}^{d}$-module in $\otclusa$. We write $\mathcal{F}$ for the $d$-exangulated category given by $\otclusa$ with this $d$-exangulated structure. This is a special case of the structure first considered in \cite{PPPP} for $d = 1$ and in \cite[Theorem 3.10]{JS} for general $d$; see also \cite[Theorem 4.4]{her_n_exang}.

\begin{proposition}\label{prop:d_exangles}
The distinguished $d$-exangles between indecomposable objects in $\mathcal{F}$ are given by \[O_{A} \to G_{d} \to G_{d - 1} \to \dots \to G_{1} \to O_{B} \to O_{A}[d]\] where $A \wr B$.
\end{proposition}
\begin{proof}
By Theorem~\ref{thm:mod_desc}(\ref{op:mod_desc:proj}), we have that $O_{C}$ is the image of a projective $A_{n}^{d}$-module in $\otclusa$ if and only if $c_{0} = 1$.
Then, by the description of $[d]$ in $\otclusa$, we have that $O_{C}$ is the shift of a projective $A_{n}^{d}$-module by $[d]$ if and only if $c_{d} = n + 2d + 1$.

We first show that \[O_{A} \to G_{d} \to G_{d - 1} \to \dots \to G_{1} \to O_{B} \to O_{A}[d]\] is a distinguished $d$-exangle in $\mathcal{F}$ if $A \wr B$. In this case we have that $O_{B} \to O_{A}[d]$ factors as $O_{B} \to O_{(b_{0}, b_{1}, \dots, b_{d - 1}, n + 2d + 1)} \to O_{A}[d]$. This is because $(B - \onevec) \wr (b_{0}, b_{1}, \dots, b_{d - 1},  n + 2d + 1)$, due to the fact that $b_{0} > a_{0} \geqslant 1$. Likewise, we have that $A - \onevec \wr (b_{0} - 1, b_{1} - 1, \dots, b_{d - 1} - 1, n + 2d)$ since $A \wr (b_{0}, b_{1}, \dots, b_{d - 1}, n + 2d + 1)$.
It can similarly be checked that Lemma~\ref{lem:clus_comp} holds. We have that $O_{(b_{0}, b_{1}, \dots, b_{d - 1}, n + 2d + 1)}$ is the image of a shifted projective. Consequently, our $(d + 2)$-angle is a distinguished $d$-exangle.

We now show that a $(d + 2)$-angle \[O_{A} \to G_{d} \to G_{d - 1} \to \dots \to G_{1} \to O_{B} \to O_{A}[d]\] with $B \wr A$ cannot be a distinguished $d$-exangle. Indeed, in this case we have $b_{d} \leqslant a_{d} - 1 < n + 2d + 1$, so the morphism $O_{B} \to O_{A}[d]$ cannot factor through $O_{C}$ with $c_{d} = n + 2d + 1$, by applying Lemma~\ref{lem:clus_comp}.
\end{proof}

We now take the ideal quotient of $\mathcal{F}$ by the morphisms which factor through a morphism $O_{1}[d] \to O_{2}$, where $O_{1}$ and $O_{2}$ are both images in $\otclusa$ of projective $A_{n}^{d}$-modules. We call the resulting category $\mathcal{G}$.

\begin{theorem} \label{thm:main2}
The category $\mathcal{G}$ is equivalent to the $d$-almost positive category $\mathcal{C}$ as a $d$-exangulated category.
\end{theorem}
\begin{proof}
We already have from Theorem~\ref{thm:clus_desc} and Theorem~\ref{thm:dcomp_desc} that the indecomposable objects of $\mathcal{G}$ and $\mathcal{C}$ are in bijection with each other. We know from Proposition~\ref{prop:d_exangles} and Theorem~\ref{thm:dcomp_desc} that these categories have the same $d$-exangles. Hence, all that needs to be shown is that $\mathcal{G}$ and $\mathcal{C}$ have the same morphisms. We claim that $\Hom_{\mathcal{G}}(O_{B}, O_{A}) \neq 0$ if and only if \[b_{0} - 1 < a_{0} < b_{1} - 1 < a_{1} < \dots < b_{d} - 1 < a_{d} < b_{0} + n + 2d.\]

Suppose first that we have $O_{A}$ and $O_{B}$ such that $\Hom_{\mathcal{O}}(O_{B}, O_{A}) \neq 0$ and $\Hom_{\mathcal{G}}(O_{B}, O_{A}) = 0$, where we abbreviate $\mathcal{O} = \otclusa$. We thus have that the morphism $O_{B} \to O_{A}$ factors through a morphism $O_{1}[d] \to O_{2}$ where $O_{1}$ and $O_{2}$ are both images in $\otclusa$ of projective $A_{n}^{d}$-modules. Hence, let $O_{1}[d] = O_{E} =  O_{(e_{0}, e_{1}, \dots, e_{d - 1}, n + 2d + 1)}$ and $O_{2} = O_{F} =  O_{(1, f_{1}, f_{2}, \dots, f_{d})}$. Since there is a morphism $O_{E} \to O_{F}$, we must have $(E - \onevec) \swr F$. As $e_{d} = n + 2d + 1$ and $f_{0} = 1$, we must have \[1 < e_{0} - 1 < f_{1} < e_{1} - 1 < \dots < f_{d - 1} < e_{d - 1} - 1 < f_{d} < n + 2d < n + 2d + 1.\] One can likewise argue that \[b_{0} - 1 < e_{0} < b_{1} - 1 < e_{1} < \dots < b_{d - 1} < e_{d - 1} < b_{d} - 1 < n + 2d + 1 < b_{0} + n + 2d\] and \[0 < a_{0} < f_{1} - 1 < a_{1} < \dots < f_{d - 1} - 1 < a_{d - 1} < f_{d} - 1 < a_{d}.\] Putting these together gives us \[b_{i} - 1 \leqslant e_{i} - 1 \leqslant f_{i + 1} - 1 < a_{i + 1}.\] Hence we must have \[a_{0} < b_{0} - 1 < a_{1} < b_{1} - 1 < \dots < a_{d} < b_{d} - 1 < a_{0} + n + 2d + 1\] rather than \[b_{0} - 1 < a_{0} < b_{1} - 1 < a_{1} < \dots < b_{d} - 1 < a_{d} < b_{0} + n + 2d,\] by considering Lemma~\ref{lem:clus_comp}.

We conversely need to show that if we have $\Hom_{\mathcal{O}}(O_{B}, O_{A}) \neq 0$ and we do not have \[b_{0} - 1 < a_{0} < b_{1} - 1 < a_{1} < \dots < b_{d} - 1 < a_{d} < b_{0} + n + 2d,\] then we must have $\Hom_{\mathcal{G}}(O_{B}, O_{A}) = 0$. Indeed, if we have $\Hom_{\mathcal{O}}(O_{B}, O_{A}) \neq 0$, then we must have $(B - \onevec) \swr A$ by Theorem~\ref{thm:clus_desc}(\ref{op:clus_desc:hom}). If we do not have \[b_{0} - 1 < a_{0} < b_{1} - 1 < a_{1} < \dots < b_{d} - 1 < a_{d} < b_{0} + n + 2d,\] then we must have \[a_{0} < b_{0} - 1 < a_{1} < b_{1} - 1 < \dots < a_{d} < b_{d} - 1 < a_{0} + n + 2d + 1.\] We then have a factorisation of $O_{B} \to O_{A}$ as $O_{B} \to O_{(b_{0}, \dots, b_{d - 1}, n + 2d + 1)} \to O_{(1, b_{0}, \dots, b_{d - 1})} \to O_{A}$ by applying Lemma~\ref{lem:clus_comp}, since $b_{0} > a_{0} + 1 \geqslant 2$ and $a_{d} < n + 2d + 1$. Thus we have $\Hom_{\mathcal{G}}(O_{B}, O_{A}) = 0$, as desired.
\end{proof}


Cluster-tilting objects in 
$\otclusa$ satisfy a $d$-exangulated analogue of the silting property when considered in~$\mathcal{F}$. Under the quotient functor $\mathcal{F} \to \mathcal{G}$ and the equivalence from Theorem~\ref{thm:main2}, they bijectively correspond to silting objects in $\mathcal{C}$. Thanks to Proposition~\ref{prop:d_exangles} and Theorem~\ref{thm:dcomp_desc}, the composition $\mathcal{F} \to \mathcal{G} \overset\sim\to \mathcal{C}$ intertwines their mutations: it intertwines mutations of cluster-tilting objects in $\otclusa$ with mutations of silting objects in $\mathcal{C}$.


\printbibliography

\end{document}